\documentclass[11pt, reqno]{amsart}

\usepackage{amssymb, amscd, latexsym}
\usepackage{enumitem}
\usepackage{amsmath}
\usepackage{amsthm}
\usepackage{bm}
\usepackage{graphicx}
\usepackage[all]{xy}
\usepackage{float}
\usepackage{mathrsfs}
\usepackage{pgf,tikz}
\usetikzlibrary{arrows}
\usepackage{array}
\usepackage{arydshln}
\usepackage[linktoc=all, hyperindex]{hyperref}
\hypersetup{
colorlinks,
linkcolor=blue,
citecolor=magenta,
anchorcolor=red,
bookmarksopen,
urlcolor=red,
filecolor=red,
}

\theoremstyle{plain}
\newtheorem{theorem}{Theorem}[section]
\newtheorem{proposition}[theorem]{Proposition}
\newtheorem{lemma}[theorem]{Lemma}
\newtheorem{corollary}[theorem]{Corollary}

\theoremstyle{definition}
\newtheorem{definition}{Definition}

\newtheorem{example}{Example}

\newtheorem{problem}{Problem}

\theoremstyle{remark}
\newtheorem*{remark}{Remark}

\def\cocoa{{\hbox{\rm C\kern-.13em o\kern-.07em C\kern-.13em o\kern-.15em A}}}

\def\C{\mathcal{C}}
\def\F{\mathcal{F}}

\def\D{\mathcal{D}}
\def\U{\mathcal{U}}
\renewcommand{\S}{\mathcal{S}}
\def\KK{\mathbb{K}}
\def\NN{\mathbb{N}}
\def\link{\mathrm{link}}

\newcommand{\gen}[1]{\left\langle #1\right\rangle}
\renewcommand{\H}{\mathcal{H}}

\newcommand{\I}{\mathcal{I}}
\newcommand{\Size}{\mathrm{Size}}

\newcommand{\wvd}{\mathrm{wd}}
\newcommand{\DS}{\mathrm{DS}}
\newcommand{\set}[1]{\left\{#1\right\}}

\begin{document}
\title[Extending simplicial complexes]{Extending simplicial complexes: Topological and combinatorial properties}

\author[M. Farrokhi D. G.]{Mohammad Farrokhi D. G.}
\email{m.farrokhi.d.g@gmail.com,\ farrokhi@iasbs.ac.ir}
\address{ Research Center for Basic Sciences and Modern Technologies (RBST), Institute for Advanced Studies in Basic Sciences (IASBS), 
Zanjan 45137-66731, Iran}
\author[A. Shamsian]{Alireza Shamsian}
\email{ashamsian@iasbs.ac.ir}
\author[A. A. Yazdan Pour]{{Ali Akbar} {Yazdan Pour}}
\email{yazdan@iasbs.ac.ir}
\address{Department of Mathematics, Institute for Advanced Studies in Basic Sciences (IASBS), Zanjan 45137-66731, Iran}
			
	
\begin{abstract}
Given an arbitrary hypergraph $\H$, we may glue to $\H$ a family of hypergraphs to get a new hypergraph $\H'$ having $\H$ as an induced subhypergraph. In this paper, we introduce three gluing techniques for which the topological and combinatorial properties (such as Cohen-Macaulayness, shellability, vertex-decomposability etc.) of the resulting hypergraph $\H'$ is under control in terms of the glued components. This enables us to construct broad classes of simplicial complexes containing a given simplicial complex as induced subcomplex satisfying nice topological and combinatorial properties. Our results will be accompanied with some interesting open problems.
\end{abstract}

\maketitle
\section*{introduction}
A \textit{simplicial complex} $\Delta$ on a vertex set $V$ is a collection of subsets of $V$ such that $\cup \Delta =V$ and $\Delta$ is closed under the operation of taking subsets. The elements of $\Delta$ are called \textit{faces} and the maximal faces of $\Delta$, under inclusion, are called the \textit{facets} of $\Delta$. A simplicial complex with facets $F_1, \ldots, F_m$ is often denoted by $\langle F_1, \ldots, F_m\rangle$. A \textit{simplex} is a simplicial complex with only one facet.

A simplicial complex $\Delta$ is called \textit{shellable} if there is a total order on facets of $\Delta$, say $F_1, \ldots, F_m$, such that $\langle F_1, \ldots, F_{i-1} \rangle \cap \langle F_i \rangle$ is generated by a non-empty set of maximal proper subsets of $F_i$ for $2 \leq i \leq m$. The notion of shellability is used to give (an inductive) proof for the Euler-Poincar\'e formula in any dimension. If $f_i$ denotes the number of $i$-faces of a $d$-dimensional polytope (with $f_{-1} = f_d = 1$), then the Euler-Poincar\'e formula states that $\sum_{i=-1}^d (-1)^if_i=1$. Shellable complexes are themselves an intermediate family among two other important families of simplicial complexes, namely vertex-decomposable and sequentially Cohen-Macaulay simplicial complexes. Indeed, we have the following implications
\begin{center}
vertex-decomposable $\Longrightarrow$ shellable $\Longrightarrow$ sequentially Cohen-Macaulay,
\end{center}
and both of these inclusions are known to be strict.

A vertex-decomposable simplicial complex $\Delta$ is defined recursively in terms of link and deletion of vertices of $\Delta$. In a more general setting, the \textit{link} and the \textit{deletion} of a face $F$ of $\Delta$ are defined as follows:
\begin{align*}
\link_{\Delta}(F) &= \{ G \in \Delta \colon \; G \cap F= \varnothing \text{ and } G \cup F \in \Delta\}, \\
\Delta \setminus F &= \{G \in \Delta \colon  \; G \cap F= \varnothing\}.
\end{align*} 
In view of the above settings, $\Delta$ is \textit{vertex-decomposable} if either it is a simplex or else there exists a vertex $v\in V$ such that 
\begin{itemize}
\item[(i)] any facet of $\Delta \setminus v$ is a facet of $\Delta$;
\item[(ii)] both complexes $\link_\Delta(v)$ and $\Delta \setminus v$ are vertex-decomposable. 
\end{itemize}

Sequentially Cohen-Macaulay complexes are defined slightly different. Let $\Delta$ be a simplicial complex on $[n]$, where $[n]=\{1,\ldots,n\}$. The \textit{pure $i$-skeleton} of $\Delta$ is the simplicial complex $\Delta^{[i]}= \langle F \in \Delta \colon \; |F|=i + 1 \rangle$. A simplicial complex $\Delta$ is \textit{Cohen-Macaulay} over $\KK$ if the \textit{Stanley-Reisner ring} $\KK[\Delta]:=S/I_\Delta$ is a Cohen-Macaulay ring, where $S=\KK[x_1,\ldots,x_n]$ is the polynomial ring with coefficients in $\KK$ and $I_\Delta=\gen{\prod_{i\in F}x_i\colon\ F\notin\Delta}$. It turns out that $\Delta$ is Cohen-Macaulay if and only if $\tilde{H}_i(\link_{\Delta}(F), \mathbb{K})=0$, for all $F \in \Delta$ and $i<\dim \link_{\Delta}(F)$ (Reisner's Theorem, see e.g. \cite[Corollary 4.2]{rs}). Consequently, as stated in \cite[Proposition 4.3]{rs}, Cohen-Macaulayness is a topological property in the sense that $\Delta$ is Cohen-Macaulay if and only if the relative singular homologies $H_i(\|\Delta\|, \|\Delta\| - p, \KK)$ of the geometric realization $\|\Delta\|$ of $\Delta$ vanish for all $i<\dim\|\Delta\|$ and $p\in\|\Delta\|$. Note that Cohen-Macaulay complexes are \textit{pure} in the sense that all of their facets have the same cardinality (see \cite[Corollary 5.1.5]{wb-jh}). Accordingly, $\Delta$ is \textit{sequentially Cohen-Macaulay} if every pure $i$-skeleton of $\Delta$ is Cohen-Macaulay, which is equivalent to say that $\KK[\Delta]$ is a sequentially Cohen-Macaulay ring (see \cite[Theorem 3.3]{amd}). Recall that a (graded) $S$-module $M$ is \textit{sequentially Cohen-Macaulay} if there exists a filtration
\begin{equation} \label{filtration}
0 = M_0 \subset M_1 \subset \cdots \subset M_r = M
\end{equation}
of (graded) submodules of $M$ such that each quotient $M_i/M_{i-1}$ is Cohen-Macaulay and
\[\dim M_1/M_0 < \dim M_2/M_1 < \cdots < \dim M_r/M_{r-1},\]
where $\dim N$ denotes the Krull dimension of $S$-module $N$.

A \textit{hypergraph} $\H$ is simply a pair $(V, E)$ of vertices $V$ and edges $E\subseteq2^V$. The \textit{independence complex} $\Delta_{\H}$ of $\H$ is the simplicial complex of all independent sets in $\H$. Clearly, every simplicial complex is the independence complex of a hypergraph. One say that $\H$ is (sequentially) Cohen-Macaulay/shellable/vertex-decomposable/pure if $\Delta_\H$ is so. In this paper, we consider a hypergraph $\H'$ obtained by gluing some hypergraphs to a central ``arbitrary'' hypergraph $\H$ and study the topological and combinatorial properties (such as Cohen-Macaulayness, shellability, vertex-decomposability etc.) of $\H'$. In this regard, Villarreal \cite[Proposition 2.2]{rhv-1} proves that the graph obtained from a graph $G$ by adding a pendant (also known as whisker) to each vertex is Cohen-Macaulay. Next Villarreal \cite[Proposition 5.4.10]{rhv-2} improves his result by showing that such graphs are pure and shellable. Later Dochtermann and Engstr\"{o}m \cite[Theorem 4.4]{ad-ae} prove that such graphs are indeed pure and vertex-decomposable. Replacing pendants with complete graphs in the Villarreal's construction, Hibi et al. \cite[Theorem 1.1]{th-ah-kk-abo} show that the resulting graph is still pure and vertex-decomposable (see also \cite{dc-un}). The idea of making small modifications to a graph in order to obtain a (sequentially) Cohen-Macaulay/shellable/vertex-decomposable graph is further explored in other papers too (see \cite{jb-avt, ad-ae, caf-hth, am-sasf-sy, mrp-sasf-sy}).

This paper is organized as follows: In the first section, we quickly review some algebraic and combinatorial backgrounds, which will be used in the sequel. Hybrid hypergraphs are introduced in Section \ref{First construction}. These hypergraphs are constructed by gluing a family of hypergraphs to a central one via a family of triples, which are assumed to satisfy the proper independence property (see the definition of PIP-triples). The main theorem of this section establishes some combinatorial/topological properties of hybrid hypergraphs and determines under which conditions a hybrid hypergraph is sequentially Cohen-Macaulay/shellable/vertex-decomposable (Theorem \ref{H'=(H,(U_i,D_i,H_i))}). It is shown that all of the results mentioned above are consequences of our main theorem of Section \ref{First construction} (see Example \ref{examples to first construction}). The idea behind the proof of Theorem \ref{H'=(H,(U_i,D_i,H_i))} is very flexible and can be applied to other suitably constructed families of hypergraphs. In Section \ref{Second and third constructions}, we present two such families of hypergraph constructions and show that the results of Theorem \ref{H'=(H,(U_i,D_i,H_i))} can be extended to these families too (Theorems \ref{(C, (C_u)_(u in U))} and \ref{(C, (C_u)_(u in U))^*}). While the gluing methods here are less general than the hybrid case, the glued components need not to satisfy the proper independence property. In the last section, we conclude the paper by proposing some research problems endowed with their motivations, which arise from our arguments in previous sections. The first four problems concern to assumptions and results of Theorem \ref{H'=(H,(U_i,D_i,H_i))} including the PIP-condition requirement of the theorem and some consequences of its part (iii), which may be of independent interest. The fifth problem, having more algebraic flavor, simply asks to study other algebraic properties and invariants of the three hypergraph constructions presented in this paper. In this regard, it is reasonable to ask whether one may control the algebraic invariants (such as Hilbert function, depth, regularity etc.) of the resulting hypergraphs in terms of the glued components. 

\section{Preliminaries} \label{prelim}
In this section, we recall basic notions of simplicial complexes, hypergraphs, and their associated ideals, which we meet in this paper.

\subsection{Simplicial complexes}
Let $\Delta$ be a simplicial complex on a vertex set $V=\{v_1, \ldots, v_n\}$. Recall that $\Delta$ is a collection of subsets of $V$, called faces of $\Delta$, such that 
\begin{itemize}
\item[(1)]$\{ v_i \} \in \Delta$ for all $1\leq i\leq n$,
\item[(2)]if $F \in \Delta$ and $G\subseteq F$, then $G\in\Delta$.
\end{itemize}
We denote by $\mathcal{F}(\Delta)$ the set of facets of $\Delta$, namely the maximal faces of $\Delta$ under inclusion. The \textit{dimension} of a face $F$, denoted by $\dim F$, is $|F| - 1$, where $|F|$ is the cardinality of $F$. Accordingly, the \textit{dimension} of $\Delta$ is defined as
$$\dim \Delta = \max\{\dim F \colon\quad F \in \Delta \}.$$
A simplicial complex is \textit{pure} if all of its facets have the same dimension. A face $F$ of $\Delta$ is \textit{maximal} with respect to a subset $W$ of $V$, or simply $F$ is \textit{$W$-maximal}, if $F\cup\{w\}\notin\Delta$ for any $w\in W\setminus F$.

Let $X$ and $Y$ be disjoint sets, and $\mathcal{X}\subseteq 2^X$ and $\mathcal{Y}\subseteq 2^Y$. The \textit{join} $\mathcal{X}\star\mathcal{Y}$ of $\mathcal{X}$ and $\mathcal{Y}$ is the subset $\{A\cup B\colon\ A\in\mathcal{X}\text{ and }B\in\mathcal{Y}\}$ of $2^{X\cup Y}$. In the case that $\mathcal{X}$ and $\mathcal{Y}$ are simplicial complexes, $\mathcal{X}\star\mathcal{Y}$ is a simplicial complex as well.

A vertex $v \in V$ is a \textit{shedding vertex} of $\Delta$ if any facet of $\Delta \setminus v$ is a facet of $\Delta$. It follows that $\Delta$ is vertex-decomposable if either it is a simplex or else it has a shedding vertex $v$ such that both $\link_\Delta(v)$ and $\Delta \setminus v$ are vertex-decomposable. Vertex-decomposability of simplicial complexes was first introduced in the pure case by Provan and Billera \cite{jsp-ljb} and then extended to arbitrary complexes by Bj\"orner and Wachs \cite[Section 11]{ab-mlw-1}.

For any non-empty set $F \subseteq \{v_{1}, \ldots, v_{n} \}$, let $\textbf{x}_F=\prod_{v_i \in F}{x_i}$, and put $\textbf{x}_\varnothing = 0$. The \textit{Stanley-Reisner ideal} of $\Delta$, denoted by $I_\Delta$, is an ideal of $\KK[x_1,\ldots,x_n]$ generated by square-free monomials $\textbf{x}_F$, where $F \not\in \Delta$.

\subsection{Hypergraphs, clutters, and their associated ideals}
Let $\H$ be a hypergraph with vertex set $V=V(\H)$ and edge set $E=E(\H)$. Following \cite{md-dp-jp}, there are two ways to remove a vertex $v$ from $\H$. The \textit{strong vertex deletion} $\H \setminus v$ is the hypergraph with vertex set $V(\H)\setminus\{v\}$ and edge set $\{e \in E(\H) \colon \; v \notin e\}$. The \textit{weak vertex deletion} $\H/v$ has the same vertex set as $\H\setminus v$ but the edge set is $\{ e \setminus \{v\} \colon \; e \in E(\H) \}$. One observe that $\H \setminus v$ deletes all edges containing $v$, while $\H/v$ removes $v$ from each edge containing it. It is straightforward to see that, if $v \neq w$ are vertices of $\H$, then:
\[(\H \setminus v) \setminus w = (\H \setminus w) \setminus v, \quad (\H/v)/w = (\H/w)/v, \quad (\H \setminus v)/ w = (\H/w) \setminus v.\]
Let $W$ be a set of vertices of $\H$ and $f\in\{0,1\}^W$ be a binary function on $W$. A \textit{$(W, f)$-deletion} of $\H$ is a hypergraph obtained from $\H$ by repeatedly strongly  deleting all vertices $w$ of $W$ with $f(w)=0$ and weakly deleting all vertices $w$ of $W$ with $f(w)=1$. The number of weak vertex deletions in a $(W,f)$-deletion $\H'$ of $\H$ is denoted by $\wvd_f(\H')$ that is $\wvd_f(\H')=\sum_{w\in W}f(w)$.

If $W$ is a set of vertices of $\H$, then the subhypergraph $\H[W]$ of $\H$ \textit{induced} on $W$ is the subhypergraph of $\H$ with vertex set $W$ and edge set $\{e\in E(\H)\colon\ e\subseteq W\}$.

For a non-empty hypergraph $\H$ on vertex set $[n]$, we define the ideal $I \left( \H \right)$ to be
$$I(\H) = \left(  \textbf{x}_T \colon \quad T \in E(\H) \right),$$
and we  set $I(\varnothing) = 0$. The ideal $I(\H)$ is called the \textit{edge ideal} of $\H$. 
Let $\Delta_{\H}$ be the simplicial complex on the vertex set $[n]$ with $I_{\Delta_{\H}} = I \left( {\H} \right)$. The simplicial complex $\Delta_{\H}$ is called the \textit{independence complex} of $\H$. Notice that $F\subseteq[n]$ belongs to $\Delta_{\H}$ if and only if it is an \textit{independent set} in $\H$, that is $e\nsubseteq F$ for every $e\in E(\H)$. The \textit{independence number} $\alpha(\H)$ of $\H$ is the maximum size of independent sets of $\H$ or equivalently $\dim\Delta_\H+1$.

A \textit{clutter} $\mathcal{C}$ with vertex set $X$ is an antichain of $2^X$ such that $X=\cup\C$. The elements of $\C$ are called \textit{circuits} of $\mathcal{C}$. A clutter $\C$ is \textit{$d$-uniform} if every circuit of $\C$ has $d$ vertices. As a hypergraph, to every clutter $\C$ one corresponds its ideal $I(\C)$. This correspondence is clearly bijective, the fact that is not valid for hypergraphs in general.

If $\mathcal{C}$ is a $d$-uniform clutter on $[n]$, then we define the \textit{complement} $\bar{\mathcal{C}}$ of $\mathcal{C}$ as
\begin{equation*}
\bar{\mathcal{C}} = \{F \subseteq [n] \colon \quad |F|=d, \,F \notin \mathcal{C}\}.
\end{equation*}

In this case, the simplicial complex $\Delta (\C)$ on the vertex set $[n]$ with $I_{\Delta (\C)} = I \left( \bar{\C} \right)$ is called the \textit{clique complex} of $\C$. A face $F \in \Delta (\C)$ is called a \textit{clique} in $\C$. It is easily seen that $F \subseteq [n]$ is a clique in $\C$ if and only if either $|F|<d$ or else all $d$-subsets of $F$ belong to $\C$.

\subsection{Criteria for (sequentially) Cohen-Macaulayness and shellability}
The recursive definition of vertex-decomposability states that a simplicial complex $\Delta$ is vertex-decomposable if $\Delta$ admits a shedding vertex $v$ such that $\link_\Delta(v)$ and $\Delta\setminus v$ are both vertex-decomposable provided that $\Delta$ is not a simplex. In the following, among other results, we show that analogous arguments work for (sequentially) Cohen-Macaulayness and shellability as well. Indeed, in the proof of our main theorems, we do not use the formal definitions of shellable or (sequentially) Cohen-Macaulay complexes as it is introduced in the introduction. Instead, the following theorem plays a crucial role in our arguments.
\begin{theorem}\label{Delta, link, and Delta * Delta'}
Let $\Delta$ and $\Delta'$ be simplicial complexes.
\begin{itemize}
\item[{\rm (i)}]If $\Delta$ is (sequentially) Cohen-Macaulay/shellable/vertex-decomposable, then so is $\link_\Delta(F)$, for every face $F$ of $\Delta$.
\item[{\rm (ii)}]If $\Delta$ has a shedding vertex $v$ such that both $\link_\Delta(v)$ and $\Delta\setminus v$ are (sequentially) Cohen-Macaulay/shellable, then so is $\Delta$.
\item[{\rm (iii)}]$\Delta\star\Delta'$ is (sequentially) Cohen-Macaulay/shellable/vertex-decomposable if and only if both $\Delta$ and $\Delta'$ are so.
\end{itemize}
\end{theorem}
\begin{proof}
(i) If $\Delta$ is Cohen-Macaulay (resp. shellable or vertex-decomposable), then by \cite[Corollary 4.2]{rs} (resp. \cite[Proposition 2.3]{jsp-ljb} or \cite[Proposition 3.7]{rw2011}), $\link_\Delta(F)$ is also Cohen-Macaulay (resp. shellable or vertex-decomposable) for every face $F$ of $\Delta$. Now, suppose $\Delta$ is sequentially Cohen-Macaulay. It is easy to check that $(\link_\Delta(F))^{[i]}=\link_{\Delta^{[i]}}(F)$ for every face $F$ of $\Delta$ and $i\geq0$. It follows that all pure skeletons of $\link_\Delta(F)$ are Cohen-Macaulay so that $\link_\Delta(F)$ is sequentially Cohen-Macaulay as well.

(ii) The result follows from \cite[Theorem 1.3]{rj-aayp} and \cite[Lemma 6]{mlw}.

(iii) The result follows from \cite[Corollary 3.3]{ab-mw-vw}, \cite[Exercise 5.1.21]{wb-jh}, \cite[Remark 10.22]{ab-mlw-2}, and \cite[Proposition 3.8]{rw2011}.
\end{proof}

In the rest of paper, we introduce three hypergraph constructions by gluing a family of hypergraphs to a given central hypergraph and examine when the resulting hypergraphs satisfy our desired properties, namely (sequentially) Cohen-Macaulayness, shellability, and vertex-decomposability.
\section{First construction} \label{First construction}
In this section we introduce our first (and main) hypergraph gluing. Under mild assumptions, i.e. the PIP-condition, we may control the topological and combinatorial properties of the resulting
hypergraphs in terms of the glued components.
\begin{definition}
Let $\H$ be a hypergraph with vertex partition $U_1\dot\cup\cdots\dot\cup U_m\dot\cup V$, and $\H_1,\ldots,\H_m$ be hypergraphs such that $\H,\H_1,\ldots,\H_m$ are pairwise disjoint. Let $ D_1,\ldots, D_m$ be sets of non-negative integers. The hypergraph with vertex set $V(\H) \cup V(\H_1) \cup \cdots \cup V(\H_m)$ and edge set
\begin{equation}
E(\H) \cup \bigcup_{i=1}^m \big\{e\cup e'\colon \quad e\subseteq U_i,\  \varnothing\neq e'\in E(\H_i),\ |e\cup e'|\in D_i\},
\end{equation}
denoted by $(\H, (U_i, D_i, \H_i)_{i=1}^m)$, is called the \textit{hybrid hypergraph} of $\H$ with respect to the gluing triples $(U_i, D_i, \H_i)_{i=1}^m$. The hypergraphs $\H_1,\ldots,\H_m$ are the \textit{glued components} of $(\H, (U_i, D_i, \H_i)_{i=1}^m)$.
\end{definition}
\begin{remark}
From the definition of hybrid hypergraphs, it is evident that $\H$ is an induced subhypergraph of $(\H, (U_i, D_i, \H_i)_{i=1}^m)$. As a result, every independent set in $\H$ is an independent set in $(\H, (U_i, D_i, \H_i)_{i=1}^m)$ as well.
\end{remark}

The independence complex of hybrid hypergraphs and their facets look wild if there is no constraints on the gluing triples $(U_i, D_i, \H_i)$. To resolve this, we apply the PIP-condition (see definition below) on gluing triples in order to describe the independence complex of hybrid hypergraphs. In what follows, we consider the \textit{Minkowski difference} $X - Y$ of two sets $X, Y$ of integers as the set $\{x - y \colon \; x \in X, \ y \in Y\}$. Also, for a hypergraph $\H$, let $\H^X$ stand for the spanning subhypergraph of $\H$ including all edges of sizes belonging to $X$, and $\Size(\H)$ denote the set $\{|e|\colon\ e\in E(\H)\}$. If $a, b$ are integers with $a\leq b$, then the set of all integers $x$ with $a\leq x\leq b$ is denoted by the interval $[a, b]$.
\begin{definition}
Let $\H$ be a hypergraph, $D$ be a set of non-negative integers, and $\alpha$ be a positive integer. The triple $(\alpha, D, \H)$ satisfies the \textit{proper independence property} (PIP) if 
\begin{itemize}
\item[(a)]every $G$ in $\F(\Delta_{\H^{D-[0, i+1]}})$ is contained properly in some $G'$ in $\F(\Delta_{\H^{D-[0, i]}})$, 
\item[(b)]every $G'$ in $\F(\Delta_{\H^{D-[0, i]}})$ contains properly some $G$ in $\F(\Delta_{\H^{D-[0, i+1]}})$, 
\end{itemize}
for all $0\leq i<\alpha$. It turns out that $[\alpha]\subseteq D - \Size(\H)$. More precisely, every $i\in[\alpha]$ belongs to $D-(\Size(\H)\setminus(D-[0,i-1]))$.
\end{definition}

If $(\alpha, D, \H)$ is any triple, then in general we have the following series of simplicial complexes
\[\Delta_{\H^{D-[0,\alpha]}}\subseteq\Delta_{\H^{D-[0,\alpha-1]}}\subseteq\cdots\subseteq\Delta_{\H^{D-[0,1]}}\subseteq\Delta_{\H^{D-[0,0]}}\]
showing that every facet $G$ of $\Delta_{\H^{D-[0,i+1]}}$ is contained in a facet $G'$ of $\Delta_{\H^{D-[0,i]}}$ for all $0\leq i<\alpha$. Being a PIP-triple indicates that not only every facet $G$ of $\Delta_{\H^{D-[0,i+1]}}$ is contained ``properly'' in a facet $G'$ of $\Delta_{\H^{D-[0,i]}}$ for all $0\leq i<\alpha$ but also every facet $G'$ of $\Delta_{\H^{D-[0,i]}}$ contains a facet $G$ of $\Delta_{\H^{D-[0,i+1]}}$ properly for all $0\leq i<\alpha$.
\begin{example}\label{Examples for PIP-triples}\ 
\begin{itemize}
\item[(i)]Let $\C$ be a $d$-uniform clutter, $D=\{d\}$, and $0<\alpha\leq d$. If $\H$ is the hypergraph induced by the edge-set $\C\cup(\gen{V(\C)}^{[d-2]}\setminus X)$, where $X\subseteq\gen{V(\C)}^{[d-\alpha-2]}$, then
\[\C'^{D-[0,i]}=\C\cup(\gen{V(\C)}^{[d-2]}\setminus \gen{V(\C)}^{[d-i-2]})\]
for all $0\leq i\leq\alpha$. It follows that $\Delta_{\C'^{D-[0,i]}}=\gen{V(\C)}^{[d-i-2]}$, for all $1\leq i\leq\alpha$. Hence $(\alpha, D, \C')$ is a PIP-triple satisfying $\C'^D=\C^D$.
\item[(ii)]Let $\S$ be a simplicial complex of dimension $d - 1$, $D=\{d\}$, and $0<\alpha<d$. Then the triple $(\alpha, D, \S)$ satisfies PIP(a) but not PIP(b) in general. Indeed, if $\S$ is a simplicial complex and $G$ is any independent set in $\S^{D-[0,i+1]}$ with $i<\alpha$, then any $G'\supset G$ with $|G'\setminus G|=1$ is an independent set in $\S^{D-[0,i]}$. On the other hand, the triple $(\alpha, \{3\}, \S)$, where $\S$ is the simplicial complex $\langle 124,134,234,235,136,127\rangle\cup\gen{45,46,47,56,57,67}$ and $\alpha>0$ does not satisfy PIP(b). To see this, we observe that the facet $123$ of $\Delta_{\S^{D-[0,0]}}$ does not contain any of the facets of $\Delta_{\S^{D-[0,1]}}=\langle 15, 26, 37\rangle$.
\item[(iii)]The triple $(\alpha, \{3\}, \H)$, where $E(\H)=\{1, 2, 3, 4, 12, 24, 34, 123\}$ and $0<\alpha<3$ satisfies PIP while it is not a simplicial complex.
\end{itemize}
\end{example}

In order to state our main result of this section, we need some preparations and preliminary lemmas.
\begin{definition}
Let $\H$ be a hypergraph. A set $D$ of vertices of $\H$ is a \textit{strong dominating set} in $\H$ if $\alpha(\H/D)=0$ or equivalently every singleton subset of $V(\H/D)$ is an edge of $\H/D$. Here by $\H/D$ we mean the hypergraph whose edge set is
\[\{ e\setminus D \colon \; e \in E(\H) \}.\]
\end{definition}
\begin{lemma}\label{F(Delta_H/v)*v subseteq F(Delta_H)}
Let $\H$ be a hypergraph and $v\in V(\H)$. Then
\begin{itemize}
\item[{\rm (i)}]$\alpha(\H\setminus v)\leq\alpha(\H)$ with equality if $v$ is a shedding vertex of $\Delta_\H$,
\item[{\rm (ii)}]$\alpha(\H/v)\leq\alpha(\H)-1$ except when $\{v\} \in E(\H)$. Indeed, if $\{v\}\notin\H$ then $\F(\Delta_{\H/v})\star\{\{v\}\}\subseteq\F(\Delta_\H)$.
\end{itemize}
\end{lemma}
\begin{proof}
(i) It is evident that $\alpha(\H\setminus v)\leq\alpha(\H)$. Suppose $v$ is a shedding vertex of $\Delta_\H$. Then every facet of $\Delta_{\H\setminus v}=\Delta_\H\setminus v$ is a facet of $\Delta_\H$. If $\alpha(\H)=|F|$ with $F$ a facet of $\Delta_\H$, then $v\notin F$ so that $F\in\Delta_{\H\setminus v}$. Thus $\alpha(\H)\leq\alpha(\H\setminus v)$, which implies that $\alpha(\H\setminus v)=\alpha(\H)$.

(ii) We show that $G\cup\{v\}\in\Delta_\H$ when $G\in\Delta_{\H/v}$. If not, $G\cup\{v\}$ contains an edge $e$ of $\H$. Then either $e\subseteq G$ or $e=e'\cup\{v\}$ for some $e' \in E(\H/v)$. In both cases, $G$ contains an edge of $\H/v$, which is a contradiction. Now, let $G\in\F(\Delta_{\H/v})$. If $G\cup\{v\}\notin\F(\Delta_\H)$, then $G\cup\{v\}\cup\{v'\}\in\Delta_\H$ for some $v'\in V(\H)\setminus(G\cup\{v\})$. As $G\cup\{v'\}\notin\Delta_{\H/v}$, $\H/v$ contains an edge $e\subseteq G\cup\{v'\}$. Then either $e \in E(\H)$ or $e=e'\setminus\{v\}$ for some $e' \in E(\H)$ containing $v$. In both cases, $G\cup\{v\}$ contains an edge of $\H$, a contradiction. Thus $G\cup\{v\}\in\F(\Delta_\H)$, from which it follows that $\F(\Delta_{\H/v})\star\{\{v\}\}\subseteq\F(\Delta_\H)$.
\end{proof}
\begin{corollary}\label{alpha(H/D)<=alpha(H)}
If $\H$ is a hypergraph and $\H'$ is a $(W,f)$-deletion of $\H$, then 
\[\alpha(\H')\leq\alpha(\H)-\wvd_f(\H')\leq\alpha(\H).\]
\end{corollary}
\begin{lemma}\label{Every hypergraph has a strong dominating independent set}
Every maximal independent set of a hypergraph is a strong dominating set.
\end{lemma}
\begin{proof}
Let $\H$ be a hypergraph and $I$ be a maximal independent set in $\H$. If $v\in V(\H)\setminus I$, then $I\cup\{v\}$ contains an edge $e$ of $\H$. Clearly, $v\in e$ so that $\{v\}\in E(\H/I)$. Thus $I$ is a strong dominating set in $\H$.
\end{proof}

According to the above settings, we are in the position to state and prove our results on the structure and combinatorial/topological properties of hybrid hypergraphs.
\begin{theorem}\label{H'=(H,(U_i,D_i,H_i))}
Let $\H'=(\H, (U_i, D_i, \H_i)_{i=1}^m)$ be a hybrid hypergraph of $\H$, where $\H$ is a hypergraph with vertex partition $U_1\dot\cup\cdots\dot\cup U_m\dot\cup V$. If $(\alpha(\H[U_i]), D_i, \H_i)_{i=1}^m$ is a family of PIP-triples, then
\begin{itemize}
\item[\rm (i)]$\dim \Delta_{\H'} = \sum_{i=1}^m\dim \Delta_{\H_i^{D_i}} +\dim \Delta_{\H[V]}+m$,
\item[\rm (ii)]$\Delta_{\H'}$ is pure if and only if $\Delta_{\H[V]}$ is pure, and $\Delta_{\H_i^{D_i-[0,s]}}$ is pure and 
\[\dim\Delta_{\H_i^{D_i-[0,s]}}-\dim\Delta_{\H_i^{D_i-[0,t]}}=t-s,\]
for all $1\leq i\leq m$ and $0\leq s\leq t\leq \alpha(\H[U_i])$,
\item[{\rm (iii)}]Let $W_i$ be a strong dominating independent set in $\H[U_i]$, for every $1\leq i\leq m$. Then $\H'$ is sequentially Cohen-Macaulay/shellable/vertex-decomposable if and only if 
\[\H[V],\quad \H_i^{D_i-[0,\alpha(\U_i)+\wvd_{f_i}(\U_i)]}\]
are so for all $1\leq i\leq m$ and $(W_i,f_i)$-deletions $\U_i$ of $\H[U_i]$.
\end{itemize}
\end{theorem}
\begin{proof}
Let $\Delta=\Delta_{\H}$ and $\Delta'=\Delta_{\H'}$ be the independence complexes of $\H$ and $\H'$, respectively. For every $i$ with $1\leq i\leq m$ and any set $X$ of integers, let  $\F(\Delta_{\H_i^X})=\{G_{i,1}^X,\ldots,G_{i,k_{i,X}}^X\}$. To each $V$-maximal set $F\in\Delta$ there corresponds a family of sets 
\[\I_{j_1, \ldots, j_m	}^F := F \cup G_{1,j_1}^{D_1-[0,a_1^F]} \cup \cdots \cup G_{m,j_m}^{D_m-[0,a_m^F]},\]
where $j_i\in\{1,\ldots,k_{i,D_i-[0,a_i^F]}\}$ and $a_i^F:=|F \cap U_i|$, for $i=1, \ldots, m$. We show that $\I_{j_1, \ldots, j_m}^F$ is a facet of $\Delta'$. First observe that $\I_{j_1, \ldots, j_m}^F$ is an independent set in $\H'$. If not, $\I_{j_1, \ldots, j_m}^F$ contains an edge $e \in E(\H')$. Notice that $e\nsubseteq F$ and hence $e\subseteq (F\cap U_i)\cup G_{i,j_i}^{D_i-[0,a_i^F]}$ for some $1\leq i\leq m$. Since $|e\cap(F\cap U_i)|\leq|F\cap U_i|=a_i^F$, it follows that 
\[\varnothing\neq e\cap G_{i,j_i}^{D_i-[0,a_i^F]}=e\cap V(\H_i) \in E(\H_i^{D_i-[0,a_i^F]}),\]
contradicting the fact that $G_{i,j_i}^{D_i-[0,a_i^F]}$ is independent in $\H_i^{D_i-[0,a_i^F]}$. Now, we show that $\I_{j_1, \ldots, j_m	}^F$ is a facet of $\Delta'$. Let $u\in U_i\setminus F$ for some $1\leq i\leq m$, and $F':=F\cup\{u\}$. Then $a_i^{F'}=a_i^F+1$. By PIP(b), the facet $G_{i,j_i}^{D_i-[0,a_i^F]}$ of $\Delta_{\H_i^{D_i-[0,a_i^F]}}$ contains a facet $G'$ of $\Delta_{\H_i^{D_i-[0,a_i^{F'}]}}$ properly. Thus $G_{i,j_i}^{D_i-[0,a_i^F]}$ contains an edge $e$ of $\H_i^{D_i-[0,a_i^{F'}]}$. If $e_u:=e\cup ((F\cup\{u\})\cap U_i)$, then $e_u$ is an edge of $\H'$ contained in $F\cup\{u\}\cup G_{i,j_i}^{D_i-[0,a_i^F]}$. On the other hand, by the definition, $\I_{j_1, \ldots, j_m}^F\cup\{u\}$ is not an independent set for any $u\in (V\setminus F)\cup\bigcup_{i=1}^m(V(\H_i)\setminus G_{i,j_i}^{D_i-[0,a_i^F]})$. This shows that $\I_{j_1, \ldots, j_m	}^F$ is a facet of $\Delta'$. 

Next we show that every facet of $\Delta'$ is of the form $\I_{j_1, \ldots, j_m}^F$, where $F \in \Delta$ is $V$-maximal. Given $G \in \Delta'$, let us define
\begin{align*}
F &:=G \cap V(\H), \\
a_i^F &:=|F \cap U_i|,&i&=1, \ldots, m, \\
G_i &:=  G \cap V(\H_i),&i&=1, \ldots, m.
\end{align*}
Since $G$ is an independent set in $\H'$, we observe that $F$ is an independent set in $\H$ and $G_i$ is independent in $\H_i^{D_i-[0,a_i^F]}$, for all $1\leq i\leq m$. Hence, for every $1\leq i\leq m$, there exists $j_i\in\{1,\ldots k_{i,D_i-[0,a_i^F]}\}$ such that $G_i \subseteq G_{i,j_i}^{D_i-[0,a_i^F]}$. If $F'\supseteq F$ is a $V$-maximal independent set of $\H$ such that $F'\cap U_i=F\cap U_i$, for $i=1,\ldots,m$, then
\begin{align*}
G= F \cup G_1 \cup \cdots \cup G_m \subseteq F' \cup G_{1,j_1}^{D_1-[0,a_1^F]} \cup \cdots \cup G_{m,j_m}^{D_m-[0,a_m^F]} = \I_{j_1, \ldots, j_m}^{F'}.
\end{align*}
Hence, $\I_{j_1, \ldots, j_m}^{F'}$ are the only facets of $\Delta'$. 

(i) For every $i$ with $1\leq i\leq m$, there exists $G_{i,j'_i}^{D_i}\in\Delta_{\H_i^{D_i}}$ (by PIP(a)) such that
\begin{align*}
|\I_{j_1, \ldots, j_r}^F| &= |F| + \sum\limits_{i=1}^m|G_{i,j_i}^{D_i-[0,a_i^F]}| \\
&= |F\cap V|+\sum\limits_{i=1}^m|(F\cap U_i)\cup G_{i,j_i}^{D_i-[0,a_i^F]}| \\
&\leq\alpha(\H[V])+\sum\limits_{i=1}^m|G_{i,j'_i}^{D_i}|\\
&\leq\dim \Delta_{\H[V]} + \sum\limits_{i=1}^m\dim \Delta_{\H_i^{D_i}} + m + 1 = |\I_{j''_1, \ldots, j''_m}^{F''}|,
\end{align*}
where $F''$ is an independent set in $\H[V]$ of maximum size and $j''_i$ is such that $G_{i,j''_i}^{D_i}$ has maximum dimension in $\Delta_{\H_i^{D_i}}$, for $i=1\ldots,m$. It follows that \[\dim \Delta' = \sum_{i=1}^m\dim \Delta_{\H_i^{D_i}} + \dim \Delta_{\H[V]} +m.\]

(ii) Clearly, $\Delta'$ is pure if and only if $|(F\cap U_i)\cup G_{i,j_i}^{D_i-[0,a_i^F]}|=|G_{i,j'_i}^{D_i}|$ for all $V$-maximal $F\in\Delta$, $i\in\{1,\ldots,m\}$, $j_i\in\{1,\ldots,k_{i,D_i-[0,a_i^F]}\}$, and $j'_i\in\{1,\ldots,k_{i,D_i}\}$. This shows that $\Delta'$ is pure if and only if $\Delta_{\H[V]}$ is pure, $\Delta_{\H_i^{D_i-[0,s]}}$ is pure, for all $0\leq s\leq \alpha(\H[U_i])$, and that
\[\dim\Delta_{\H_i^{D_i-[0,s]}}-\dim\Delta_{\H_i^{D_i-[0,t]}}=t-s,\]
for all $0\leq s\leq t\leq \alpha(\H[U_i])$.

(iii) Let $u\in U_l$ ($1\leq l\leq m$) be such that $\{u\}\notin\H$. Put $U'_i=U_i\setminus\{u\}$ and $D'_i=D_i-[0,\delta_{i,l}]$, for $i=1,\ldots,m$. Clearly, 
\[\Delta'\setminus u=\gen{\I_{j_1, \ldots, j_r}^F\colon\ F\subseteq V(\H)\setminus\{u\}\ \text{is $V$-maximal in}\ \Delta_\H}\]
by PIP(a), from which we conclude that $u$ is a shedding vertex of $\Delta'$. A simple verification shows that
\begin{equation}\label{Delta'-u}
\Delta'\setminus u=\Delta_{(\H\setminus u, (U'_i, D_i, \H_i)_{i=1}^m)}.
\end{equation}
On the other hand,
\[\link_{\Delta'}(u)=\gen{\I_{j_1,\ldots,j_m}^F\setminus\{u\}\colon\ u\in F\subseteq V(\H)\ \text{is $V$-maximal in}\ \Delta_\H}.\]
We show that
\begin{equation}\label{link_Delta'(u)}
\link_{\Delta'}(u)=\Delta_{(\H/u, (U'_i, D'_i, \H_i)_{i=1}^m)}.
\end{equation}
Let $\H^*:=(\H/u, (U'_i, D'_i, \H_i)_{i=1}^m)$. If $F\in\link_{\Delta'}(u)$ and $F\notin\Delta_{\H^*}$, then $F	$ contains an edge $e_1\cup e_2$ of $\H^*$, where $e_1\subseteq U_l\setminus\{u\}$, $\varnothing\neq e_2 \in E(\H_l)$, and $|e_1\cup e_2|\in D'_l=D_l-[0,1]$. Then $(e_1\cup\{u\})\cup e_2$ is an edge of $\H'$ contained in $F\cup\{u\}$, which is a contradiction. Thus $\link_{\Delta'}(u)\subseteq\Delta_{\H^*}$. Conversely, let $F\in\Delta_{\H^*}$. If $F\notin\Delta'$, then $F$ contains an edge $e_1\cup e_2$ of $\H'$, where $e_1\subseteq U_l$, $\varnothing\neq e_2 \in E(\H_l)$, and $|e_1\cup e_2|\in D_l$. This yields that $e_1\cup e_2\subseteq F$ is an edge of $\H^*$, which contradicts our assumption. Thus $F\in\Delta'$. To show that $F\in\link_{\Delta'}(u)$, we prove $F\cup\{u\}\in\Delta'$. Suppose on the contrary that $F\cup\{u\}$ contains an edge $e_1\cup e_2$ of $\H'$, where $e_1\subseteq U_l$, $\varnothing\neq e_2 \in E(\H_l)$, and $|e_1\cup e_2|\in D_l$. Clearly, $u\in e_1$ and $(e_1\setminus\{u\})\cup e_2\subseteq F$ is an edge of $\H^*$, which is a contradiction. Thus $F\in\link_{\Delta'}(u)$. This shows the equality in \eqref{link_Delta'(u)}.

Let $W_i=\{u_i^1,\ldots,u_i^{k_i}\}\subseteq U_i$ be a strong dominating independent set of $\H[U_i]$, for $i=1,\ldots,m$ (see Lemma \ref{Every hypergraph has a strong dominating independent set}). First assume that $\H[V]$ and $\H_i^{D_i-[0,\alpha(\U_i)+\wvd_{f_i}(\U_i)]}$ are sequentially Cohen-Macaulay, for all $i=1,\ldots,m$ and $(W_i,f_i)$-deletions $\U_i$ of $\H[U_i]$. Let $1\leq j\leq m$ be such that $k_j>0$ and put $W'_j=W_j\setminus\{u_j^1\}$. It is obvious that $W'_j$ is a strong dominating independent set in both $\H[U_j]\setminus u_j^1$ and $\H[U_j]/u_j^1$.

From Corollary \ref{alpha(H/D)<=alpha(H)}, it is evident that $\Delta'\setminus u_j^1=\Delta_{(\H\setminus u_j^1,(U'_i,D_i,\H_i)_{i=1}^m)}$ is sequentially Cohen-Macaulay because for all $(W'_j,f'_j)$-deletion $\U'_j$ of $(\H\setminus u_j^1)[U'_j]$,
\begin{align*}
\Delta_{\H_j^{D_j-[0,\alpha(\U'_j)+\wvd_{f'_j}(\U'_j)]}}&=\Delta_{\H_j^{D_j-[0,\alpha(\U_j)+\wvd_{f_j}(\U_j)]}}
\end{align*}
is sequentially Cohen-Macaulay by assumption, where $\U_j$ is the $(W_j,f_j)$-deletion of $\H[U_j]$ with $f_j(u_j^1)=0$ and $f_j|_{W'_j}=f'_j$.

On the other hand, $\link_{\Delta'}(u_j^1)=\Delta_{(\H/u_j^1, (U'_i, D'_i, \H_i)_{i=1}^m)}$. Let $\U'_j$ be a $(W'_j,f'_j)$-deletion of $(\H/u_j^1)[U'_j]$, and $\U_j$ be the $(W_j,f_j)$-deletion of $\H[U_j]$, where $f_j(u_j^1)=1$ and $f_j|_{W'_j}=f'_j$. Since
\begin{align*}
\Delta_{\H_j^{D'_j-[0,\alpha(\U'_j)+\wvd_{f'_j}(\U'_j)]}}&=\Delta_{\H_j^{D_j-[0,1]-[0,\alpha(\U'_j)+\wvd_{f'_j}(\U'_j)]}}\\
&=\Delta_{\H_j^{D_j-[0,\alpha(\U_j)+\wvd_{f_j}(\U_j)]}}
\end{align*}
is sequentially Cohen-Macaulay by assumption for any such $\U'_j$, we conclude that $\link_{\Delta'}(u_j^1)$ is sequentially Cohen-Macaulay too. Therefore, $\Delta'$ is sequentially Cohen-Macaulay by Theorem \ref{Delta, link, and Delta * Delta'}(ii) in this case. Finally, assume that $k_1=\cdots=k_m=0$. Then 
\[\Delta'=\Delta_{\H[V]}\star\Delta_{\H_1^{D_1}}\star\cdots\star\Delta_{\H_i^{D_m}}.\]
Notice that $\alpha(\H[U_i])=0$ so that $\Delta_{\H_i^{D_i}}=\Delta_{\H_i^{D_i-[0,\alpha(\H[U_i])]}}$ is sequentially Cohen-Macaulay for $i=1,\ldots,m$ by assumption. Applying Theorem \ref{Delta, link, and Delta * Delta'}(iii), it follows that $\Delta'$ is sequentially Cohen-Macaulay. 

Conversely, suppose that $\Delta'$ is sequentially Cohen-Macaulay. By Theorem \ref{Delta, link, and Delta * Delta'}(i), we get $\Delta_{\H[V]}=\Delta'[V]=\link_{\Delta'}(\bigcup_{i=1}^mG_{i,1}^{D_i})$ is sequentially Cohen-Macaulay. On the other hand, if $G_i\in\Delta_{\H[U_i]}$ ($1\leq i\leq m$) and $F$ is a $V$-maximal independent set in $\H$ such that $F\cap U_i=G_i$, then
\begin{equation}\label{link_Delta'(G)=Delta((H_U_i,(U_i,D_i,H_i)))}
G_i\in\link_{\Delta'}(G)=\Delta_{(\H[U_i],(U_i,D_i,\H_i))},
\end{equation}
where $G=\I_{1,\ldots,1}^F\setminus (G_i\cup G_{i,1}^{D_i-[0,|G_i|]})$. Thus, if $G_i$ is a maximal independent set in $\H[U_i]$, then
\[\link_{\Delta'}(G\cup G_i)=\link_{\link_{\Delta'}(G)}(G_i)=\Delta_{\H_i^{D_i-[0,|G_i|]}}\]
is sequentially Cohen-Macaulay by Theorem \ref{Delta, link, and Delta * Delta'}(i). In particular, $\Delta_{\H_i^{D_i-[0,\alpha(\H[U_i])]}}$ is sequentially Cohen-Macaulay. Let $\U_i$ be a $(W_i,f_i)$-deletion of $\H[U_i]$. Put $G_i=\{w\in W\colon\ f_i(w)=1\}$. By \eqref{link_Delta'(u)} and \eqref{link_Delta'(G)=Delta((H_U_i,(U_i,D_i,H_i)))}, we observe that
\[\link_{\Delta'}(G\cup G_i)=\Delta_{(\H[U_i]/G_i,(U_i\setminus G_i,D_i-[0,\wvd_{f_i}(\U_i)],\H_i))}\]
is sequentially Cohen-Macaulay, which implies that $\Delta_{\H_i^{D_i-[0,\alpha(\U_i)+\wvd_{f_i}(\U_i)]}}$ is sequentially Cohen-Macaulay as well.

The cases of shellability and vertex-decomposability follow from the same discussions as above.
\end{proof}

\begin{corollary}\label{Hybrid hypergraphs of designs}
Let $\H$ be a hypergraph with vertex partition $U_1\dot\cup\cdots\dot\cup U_m$ and $d_i=\alpha(\H[U_i])$, for $i=1,\ldots,m$. Let $\C_i$ be a $d_i$-uniform clutter and $\H_i$ be the hypergraph induced by the edge-set $\C_i\cup\gen{V(\C_i)}^{[d_i-2]}$ for $i=1,\ldots,m$. Let $\H'$ be the hybrid hypergraph $(\H,(U_i,\{d_i\},\H_i)_{i=1}^m)$. Then 
\begin{itemize}
\item[\rm (i)]$\dim\Delta_{\H'} = \sum_{i=1}^m\dim \Delta_{\C_i} + m - 1$,
\item[\rm (ii)]$\Delta_{\H'}$ is pure if and only if $\C_i=\binom{V(\C_i)}{d_i}$ is complete $d_i$-clutter, for all $i=1,\ldots,m$,
\item[\rm (iii)]$\H'$ is vertex-decomposable if and only if $\C_i$ is vertex-decomposable for all $1\leq i\leq m$ such that $\H[U_i]$ has a unique maximal independent set,
\end{itemize}
\end{corollary}
\begin{proof}
By Example \ref{Examples for PIP-triples}(i), $(d_i,\{d_i\},\H_i)$ is a PIP-triple for all $1\leq i\leq m$. To prove part (iii), let $W_i$ be a strong dominating independent set in $\H[U_i]$ and $\U_i$ be a $(W_i,f_i)$-deletion of $\H[U_i]$. One observe that $\alpha(\U_i)+\wvd_{f_i}(\U_i)=0$ if and only if $W_i$ is the unique maximal independent set of $\H[U_i]$ and $f_i\equiv0$. Since  $\Delta_{\H_i^{\{d_i\}-[0,j]}}=\gen{V(\C_i)}^{[d-j-2]}$ is vertex-decomposable for all $1\leq i\leq m$ and $1\leq j\leq d_i$, Theorem \ref{H'=(H,(U_i,D_i,H_i))}(iii) states that $\H'$ is vertex-decomposable if and only if $\Delta_{\H_i^{\{d_i\}-[0,0]}}=\Delta_{\C_i}$ is vertex-decomposable for all $1\leq i\leq m$ such that $\H[U_i]$ has a unique maximal independent set. Parts (i) and (ii) follow immediately from Theorem \ref{H'=(H,(U_i,D_i,H_i))}(i,ii).
\end{proof}
\begin{corollary}\label{C'=C U (e in U_i U V(H_i): |e|=d and e cap V(H_i) in H_i)}
Let $\C$ be a $d$-uniform clutter, $U_1, \ldots, U_m$ be a clique partition of $\C\setminus V$ for some subset $V$ of $V(\C)$, and let $\H_1,\ldots,\H_m$ be hypergraphs such that the triples $(\alpha(\C[U_i]),\{d\},\H_i)$ satisfy the PIP-conditions, for $i=1,\ldots,m$. Let $\C'$ be the $d$-uniform clutter defined as
\[\C'=\C\cup\bigcup_{i=1}^m\{e\subseteq U_i\cup V(\H_i)\colon\quad |e|=d\emph{ and }e\cap V(\H_i) \in E(\H_i)\}.\]
Then
\begin{itemize}
\item[\rm (i)]$\dim \Delta_{\C'} = \sum_{i=1}^m\dim \Delta_{\H_i^{\{d\}}} + \dim \Delta_{\C[V]} + m$,
\item[\rm (ii)]$\Delta_{\C'}$ is pure if and only if $\Delta_{\C[V]}$ is pure, and $\Delta_{\H_i^{[s,d]}}$ is pure and 
\[\dim\Delta_{\H_i^{[t,d]}}-\dim\Delta_{\H_i^{[s,d]}}=t-s,\]
for all $i=1,\ldots,m$ and $\max\{d-|U_i|,1\}\leq s\leq t\leq d$,
\item[\rm (iii)]$\C'$ is sequentially Cohen-Macaulay/shellable/vertex-decomposable if and only if $\C[V]$ and $\H_i^{[d-j, d]}$ are so for all $j\in J$, where $J=\{\beta,\ldots,\alpha(\C[U_i])\}$ with 
\[\beta=\begin{cases}
0,&|U_i|<d,\\
\min\{|U_i|-(d-1), d-1\},&|U_i|\geq d.
\end{cases}\]
\end{itemize}
\end{corollary}
\begin{proof}
It is easy to verify that $\C'=(\C,(U_i,\{d\},\H_i)_{i=1}^m)$. Let $1\leq i\leq m$. Clearly, $\alpha(\C[U_i])=\min\{|U_i|,d-1\}$. Let $W$ be a strong dominating independent in $\C[U_i]$. Notice that a subset $\{u_1,\ldots,u_k\}$ of $U_i$ is a strong dominating independent set in $\C[U_i]$ if and only if $k=\alpha(\C[U_i])$. Also, for every subset $X$ of $U_i$ and $x\in X$, we have $\alpha(\C[X]/x)=\alpha(\C[X]) - 1$ and $\alpha(\C[X]\setminus x)=\alpha(\C[X]) - \delta_{|U_i|<d}$. It follows that the set of all $\alpha(\U_i)+\wvd_f(\U_i)$ with $\U_i$ a $(W,f)$-deletion of $\C[U_i]$ coincides with the set $J$. Now Theorem \ref{H'=(H,(U_i,D_i,H_i))} yields the desired conclusion.
\end{proof}
\begin{corollary}\label{C'=C U binom(U_i U V(H_i, d)}
Let $\C$ be a $d$-uniform clutter, $U_1, \ldots, U_m$ be a clique partition of $\C$, and let $\H_1,\ldots,\H_m$ be disjoint simplexes of dimensions at least $d-2$. If
\[\C'=\C\cup\bigcup_{i=1}^m\binom{U_i\cup V(\H_i)}{d},\]
then
\begin{itemize}
\item[\rm (i)]$\Delta_{\C'}$ is pure vertex-decomposable of dimension $(d-1)m-1$, and
\item[\rm (ii)]the ring $\KK[V(\C')]/I(\C')$ is Cohen-Macaulay of dimension $(d-1)m$.	
\end{itemize}
\end{corollary}
\begin{proof}
First observe that $(\alpha(\C[U_i]),\{d\},\H_i)_{i=1}^m$ is a family of PIP-triples for $\dim \H_i \geq d-2$ viewed as a simplicial complex and $\Delta_{\H_i^{\{d\}-[0,i]}}=\gen{V(\H_i)}^{[d-i-2]}$, for all $i=0,\ldots,\alpha(\C[U_i])$. In view of Corollary \ref{C'=C U (e in U_i U V(H_i): |e|=d and e cap V(H_i) in H_i)}(ii), $\Delta_{\C'}$ is pure and vertex-decomposable of dimension $(d-1)m-1$. Part (ii) follows from the more general fact that pure vertex-decomposable simplicial complexes are Cohen-Macaulay and that $\dim\KK[\Delta_{\C'}]=1+\dim\Delta_{C'}$ (see \cite[Theorem 5.1.4]{wb-jh}).
\end{proof}

Theorem \ref{H'=(H,(U_i,D_i,H_i))} and its related corollaries establish alternate proofs for previously known results we address here.
\begin{example}\label{examples to first construction}\ 
\begin{itemize}
\item[(i)]Let $G$ be a graph with vertex set $V(G)=\{v_1,\ldots,v_n\}$. The corona graph $G'$ of $G$ is a graph obtained from $G$ by attaching a pendant to each vertex of $G$, that is $G'$ is the graph with vertex set $V(G)\cup\{w_1,\ldots,w_n\}$ and edge set $E(G)\cup\{v_iw_i\colon\ 1\leq i\leq n\}$. It is shown by Villarreal \cite[Proposition 2.2]{rhv-1} that the graph $G'$ is Cohen-Macaulay. Villarreal \cite[Proposition 5.4.10]{rhv-2} improves his result by showing that $G'$ is pure and shellable. Later Dochtermann and Engstr\"{o}m \cite[Theorem 4.4]{ad-ae} prove that $G'$ is indeed pure and vertex-decomposable. Hibi, Higashitani, Kimura, and O'Keefe \cite{th-ah-kk-abo} and Cook II and Nagel \cite{dc-un} give two generalizations of Villarreal's construction. In \cite[Theorem 1.1]{th-ah-kk-abo}, Hibi et. al. show that the graph obtained from identification of every vertex $v$ of $G$ with a vertex of a complete graph $G_v$ is still pure and vertex-decomposable. Cook II and Nagel \cite[Theorem 3.3 and Corollary 3.5]{dc-un} apply a different generalization and show that $\Delta_{G^\pi}$ is vertex-decomposable and Cohen-Macaulay if $G$ is a graph and $G^\pi$ is the graph obtained from $G$ with a clique partition $\pi=\{W_1,\ldots,W_t\}$ as follows: $V(G^\pi)=V(G)\cup\{w_1,\ldots,w_t\}$ for some distinct vertices $w_1,\ldots,w_t$ not in $G$, and $E(G^\pi)=E(G)\cup\{vw_i\colon\ v\in W_i\}$. All of these results and consequences thereafter are special cases of Corollary \ref{C'=C U binom(U_i U V(H_i, d)}.
\item[(ii)]In \cite[Proposition 4.3]{ad-ae} the authors show that if $G_r$ is the graph obtained from an $r$-cycle with attaching a new vertex to two adjacent vertices of the cycle, then $I(G_r)$ is vertex-decomposable and hence sequentially Cohen-Macaulay. This follows simply from Corollary \ref{C'=C U (e in U_i U V(H_i): |e|=d and e cap V(H_i) in H_i)}.
\item[(iii)]Let $G$ be a chordal graph that is $G$ has no induced cycles of length greater than $3$. In \cite[Theorem 3.2]{caf-avt}, the authors show that $\Delta_G$ is sequentially Cohen-Macaulay. Later, in \cite[Theorem 2.13]{avt-rhv} it is shown that $\Delta_G$ is indeed shellable. This result is also strengthened by Woodroofe \cite[Corollary 7(2)]{rw2009} (and independently by Dochtermann and Engstr\"{o}m \cite[Theorem 4.1]{ad-ae}) by showing that $\Delta_G$ is vertex-decomposable. We use our method to obtain the mentioned results. First observe that the chordal graph $G$ has a vertex $v$ with complete neighborhood $U$ (see \cite{gad}). If $V:=V(G)\setminus(U\cup\{v\})$, then $(G\setminus v, (U, \{2\}, \gen{v}))=G$. Hence, an inductive argument in conjunction with Corollary \ref{C'=C U (e in U_i U V(H_i): |e|=d and e cap V(H_i) in H_i)} shows that $\Delta_G$ is vertex-decomposable.
\item[(iv)]Let $\Delta$ be a simplicial complex on the vertex set $V=\{v_1, \ldots, v_n\}$. Following \cite{jb-avt}, an $m$-coloring $\chi$ of $\Delta$ is a partition $V = V_1 \cup \cdots \cup V_m$ of the vertices (where the sets $V_i$ are allowed to be empty) such that $|F \cap V_i| \leq 1$ for all $F\in\Delta$ and $1 \leq i \leq m$. For such a coloring $\chi$ of $\Delta$, define the simplicial complex $\Delta_\chi$ on the vertex set $\{v_1, \ldots, v_n, w_1, \ldots, w_m\}$ with faces $\sigma \cup \tau$ where $\sigma \in \Delta$ and $\tau$ is any subset  of $\{w_i\colon\ \sigma\cap V_i=\varnothing\}$. In \cite[Theorem 7]{jb-avt}, it is shown that $\Delta_\chi$ is pure and vertex-decomposable. In the following we show that this result is an immediate consequence of Theorem~\ref{H'=(H,(U_i,D_i,H_i))}. Let $\C$ and $\C'$ be the clutters with $\Delta=\Delta_{\C}$ and $\Delta_\chi=\Delta_{\C'}$. It is easy to see that $\C'=(\C,(V_i, \{2\}, \gen{w_i})_{i=1}^m)$. Suppose without loss of generality that $V_1,\ldots,V_m$ are non-empty. It follows from the definition of $m$-coloring that $\alpha(\C[V_i]) = 1$, hence $(\alpha(\C[V_i]), \{2\}, \gen{w_i})$ is a PIP-triple for all $1\leq i\leq m$. Now from Theorem \ref{H'=(H,(U_i,D_i,H_i))} we conclude that $\Delta_\chi=\Delta_{\C'}$ is pure and vertex-decomposable.
\end{itemize}
\end{example}
\section{Second and third constructions} \label{Second and third constructions}
The aim of this section is to give yet two more gluing techniques leading us to similar results as in Theorem \ref{H'=(H,(U_i,D_i,H_i))}. Though our constructions here coincide with the first one in very special cases, they are independent from the first construction in general. In what follows, our gluing processes are applied to clutters instead of hypergraphs and that every component is glued via a single vertex to the central clutter. The main idea of our proofs obey from that in Theorem \ref{H'=(H,(U_i,D_i,H_i))}. The following definition is used in order to describe the purity of the resulting clutters.
\begin{definition}
Let $\C$ be a clutter and $\U$ be an induced subclutter of $\C$. Then $\U$ is \textit{independently embedded} in $\C$ if $F\cup X\in\F(\Delta_\C)$ with $F\subseteq V(\U)$ and $X\subseteq V(\C)\setminus V(\U)$ implies $F\in\F(\Delta_\U)$.
\end{definition}
\begin{remark}
Let $\C$ be a clutter and $v\in V(\C)$. Viewing $\C$ as a hypergraph, the weakly deletion $\C/v$ is a hypergraph but not a clutter in general. However, the set $\min(\C/v)$ of all minimal elements of $\C/v$ under inclusion is a clutter whose edge ideal is the same as that of $\C/v$. This is usually referred to hypergraph reduction of $\C/v$.
\end{remark}
\begin{theorem}\label{(C, (C_u)_(u in U))}
Let $\C$ be a clutter with vertex set $U\dot\cup V$ and $\{\C_u\}_{u\in U}$ be a family of non-empty clutters such that $\C$ and $\{\C_u\}_{u\in U}$ are pairwise disjoint. Let $(\C,\{\C_u\}_{u\in U})$ be the clutter obtained from $\C$ as follows:
\[ (\C,\{\C_u\}_{u\in U}) = \C \cup \bigcup_{u\in U}\{e\cup\{u\}\colon\quad e\in \C_u\}.\]
If $\C':=(\C,\{\C_u\}_{u\in U})$ and $\Delta':=\Delta_{\C'}$, then
\begin{itemize}
\item[{\rm (i)}]$\dim \Delta'  = \sum_{u\in U}|V(\C_u)| + \dim \Delta_{\C[V]}$,
\item[{\rm (ii)}]$\Delta'$ is pure if and only if $|\C_u|=1$ for all $u\in U$, $\Delta_{\C[V]}$ is pure, and $\C[V]$ is independently embedded in $\C$,
\item[{\rm (iii)}]$\C'$ is sequentially Cohen-Macaulay/shellable/vertex-decomposable if and only if $\C[V]$ and $\C_u$ are so for all $u\in U$,
\item[{\rm (iv)}]$\C'$ is Cohen-Macaulay if and only if it is pure and $\C[V]$ is Cohen-Macaulay.
\end{itemize}
\end{theorem}

\begin{proof}
Let $\Delta=\Delta_{\C}$ be the independence complex of $\C$. For every $V$-maximal face $F \in \Delta$, we construct a block of facets of $\Delta'$ as follows:

Given a $V$-maximal face $F \in \Delta$ and $u\in U$, let 
\begin{equation} \label{structure of G_i,j in (C, {C_u}_{u in U})}
\bigg\{ G_{u,1}^F, \ldots, G_{u, k_u^F}^F \bigg\} =\begin{cases}
\{V(\C_u)\}, & \text{ if } u \notin F, \\\\
\F(\Delta_{\C_u}), & \text{ if } u \in F.
\end{cases}
\end{equation}
Then, we consider the block of sets associated to $F$ defined by
\[H_{j_U}^F := F \cup \bigcup_{u\in U}G_{u,j_u}^F,\]
where $j_U=(j_u\colon\ u\in U)$ and $j_u \in \{1,\ldots,k_u^F\}$, for all $u\in U$. We claim that every facet of $\Delta'$ is of the form $H_{j_U}^F$ for some $V$-maximal face $F\in\Delta$. First of all, note that the sets $H_{j_U}^F$ mentioned above are facets of $\Delta'$. In order to prove our claim, it is enough to show that every face $G \in \Delta'$ is contained in some $H_{j_U}^F$. 

Given $G \in \Delta'$, let us define
\begin{align*}	
F_U &:= G \cap U, \\
F_V &:= G \cap V, \\
G_u &:= G \cap V(\C_u), \quad u\in U.
\end{align*}
Since $G$ is an independent set in $\C'$, we conclude that $F:=F_U\cup F_V$ is an independent set in $\C$ and $G_u \neq V(\C_u)$ if $u \in F$. Hence our construction of $G_{u,1}^F,\ldots,G_{u,k_u^F}^F$ as in \eqref{structure of G_i,j in (C, {C_u}_{u in U})} guarantees the existence of $j_u \in \{1,\ldots,k_u^F\}$ such that $G_u \subseteq G_{u,j_u}^F$ for all $u\in U$. Then
\[G= F \cup \bigcup_{u\in U}G_u\subseteq F \cup \bigcup_{u\in U}G_{u,j_u}^F \subseteq H_{j_U}^{F'},\]
where $F'\supseteq F$ is a $V$-maximal face of $\Delta$ satisfying $F'\cap U=F_U$.

(i) Note that 
\begin{align*}
|H_{j_U}^F| &= |F| + \sum_{u\in U}|G_{u,j_u}^F| \\
&= \sum_{u\in U\cap F}|G_{u,j_u}^F\cup\{u\}|+\sum_{u\in U\setminus F}|V(\C_u)|+|F\cap V|\\
&\leq \sum_{u\in U}|V(\C_u)| + \dim \Delta_{\C[V]} + 1 = |H_{j_U}^{F'}|,
\end{align*}
where $F'$ is a facet of maximum dimension in $\Delta_{\C[V]}$. It follows that 
\[\dim \Delta' = \sum_{u\in U}|V(\C_u)| + \dim \Delta_{\C[V]}.\]

(ii) Suppose $\Delta'$ is pure. It is evident that $\Delta_{\C_u}$ is pure of dimension $|V(\C_u)|-2$ for all $u\in U$, $\Delta_{\C[V]}$ is pure, and $\C[V]$ is independently embedded in $\C$. Let $u\in U$. To show that $|\C_u|=1$ suppose on the contrary that $\C_u$ has two distinct circuits $e_1$ and $e_2$. Let $F$ be a facet of $\Delta_{\C_u}$ containing $e_1\cap e_2$. Notice that $e_1\cap e_2$ is independent in $\C_u$ as $\C_u$ is a clutter. Then $|F|\leq|V(\C_u)|-2$ for $F$ misses a vertex of $e_1$ and a vertex of $e_2$. This contradicts the fact that $\Delta_{\C_u}$ is pure of dimension $|V(\C_u)|-2$. The converse is obvious.

(iii) Let $u\in U$. Then 
\begin{align}
\Delta'\setminus u&=\gen{V(\C_u)\cup H_{j_{U\setminus\{u\}}}^F\colon\ F\subseteq V(\C)\setminus\{u\}\ \text{is $V$-maximal in}\ \Delta_\C}\notag\\
&=\gen{V(\C_u)}\star\Delta_{(\C\setminus u, \{\C_{u'}\}_{u'\in U\setminus\{u\}})}.\label{Delta'-u2}
\end{align}
This implies at once that every $u\in U$ is a shedding vertex of $\Delta'$ as every facet of $\Delta'\setminus u$ is also a facet of $\Delta'$.

On the other hand,
\begin{align}
\link_{\Delta'}(u)&=\gen{H_{j_{U\setminus\{u\}}}^F\setminus\{u\}\colon\ u\in F\subseteq V(\C)\ \text{is $V$-maximal in}\ \Delta_\C}\notag\\
&=\Delta_{\C_u}\star\Delta_{(\D_u,\{\C_{u'}\}_{u'\in U})},\label{link_Delta'(u)2}
\end{align}
where $\D_u=\min(\C / u)$.

If $\Delta'$ is sequentially Cohen-Macaulay, then by Theorem \ref{Delta, link, and Delta * Delta'}(i), $\link_{\Delta'}(u)$ and hence $\C_u$ is sequentially Cohen-Macaulay for all $u\in U$. Furthermore, $\link_{\Delta'}(\cup_{u\in U}V(\C_u))=\Delta_{\C[V]}$, which implies that $\C[V]$ is sequentially Cohen-Macaulay by applying Theorem \ref{Delta, link, and Delta * Delta'}(i) once more. To prove the converse, assume that $\C[V]$ and $\C_u$ ($u\in U$) are sequentially Cohen-Macaulay. Induction on $|U|$, Theorem \ref{Delta, link, and Delta * Delta'}, and equations \eqref{Delta'-u2} and \eqref{link_Delta'(u)2}, yield the desired conclusion. The cases of shellability and vertex-decomposability is entirely the same as above.

(iv) Suppose $\C'$ is Cohen-Macaulay. Then $\Delta'$ is pure and the link of any face is Cohen-Macaulay (see \cite[Corollary 5.1.5]{wb-jh} and Theorem \ref{Delta, link, and Delta * Delta'}(i)). It follows that $\link_{\Delta'}(\cup_{u\in U}V(\C_u))=\Delta_{\C[V]}$ is Cohen-Macaulay. Conversely, if $\C[V]$ is Cohen-Macaulay and $\Delta'$ is pure, then by part (i), $|\C_u|=1$ for all $u\in U$. It follows that $\C[V]$ and $\C_u$ ($u\in U$) are (sequentially) Cohen-Macaulay. From (iii), we conclude that $\C'$ is sequentially Cohen-Macaulay. Since $\Delta'$ is pure, we get the desired conclusion.
\end{proof}
\begin{corollary}[{Compare with \cite[Theorem 8.2]{sf}}]
Let $\C$ be a clutter on $[n]$ and $C_1,\ldots,C_n$ be non-empty sets such that $V(\C),C_1,\ldots,C_n$ are pairwise disjoint. Let 
\[\C':=\C\cup\{C_i\cup\{i\}\colon\ i\in[n]\}.\]
Then $\Delta_{\C'}$ is a pure and vertex-decomposable simplicial complex, hence Cohen-Macaulay.
\end{corollary}

One observe that the above corollary covers the results of Villarreal and Dochtermann-Engstr\"{o}m in Example \ref{examples to first construction}(i).
\begin{theorem}\label{(C, (C_u)_(u in U))^*}
Let $\C$ be a clutter with vertex set $U\dot\cup V$ and $\{\C_u\}_{u\in U}$ be a family of non-empty clutters such that $\C$ and $\{\C_u\}_{u\in U}$ are pairwise disjoint. Let $(\C,\{\C_u\}_{u\in U})^*$ be the clutter obtained from $\C$ as follows:
\[ (\C,\{\C_u\}_{u\in U})^* = \C \cup \bigcup_{u\in U}\C_u \cup \bigcup_{u\in U}\{\{u\}\}\star\C_u^*,\]
where $\C_u^*$ is the set of minimal elements of the set $\{e \setminus x\colon\ x\in e,\ e\in\C_u\}$ with respect to inclusion, for all $u\in U$. If $\C':=(\C,\{\C_u\}_{u\in U})^*$ and $\Delta':=\Delta_{\C'}$, then
\begin{itemize}
\item[{\rm (i)}]$\dim \Delta'  = |U| + \sum_{u\in U}|\Delta_{\C_u}| + \dim \Delta_{\C[V]}$,
\item[{\rm (ii)}]$\Delta'$ is pure if and only if $\Delta_{\C[V]}$ is pure, $\C[V]$ is independently embedded in $\C$, $\Delta_{\C_u}$ and $\Delta_{\C_u^*}$ are pure and $\dim\Delta_{\C_u^*} = \dim\Delta_{\C_u} - 1$ for all $u\in U$,
\item[{\rm (iii)}]$\C'$ is sequentially Cohen-Macaulay/shellable/vertex-decomposable if and only if $\C[V]$, and $\C_u$ and $\C_u^*$ are so for all $u\in U$.
\end{itemize}
\end{theorem}

\begin{proof}
Let $\Delta=\Delta_{\C}$ be the independence complex of $\C$. For every $V$-maximal face $F \in \Delta$, we construct a block of facets of $\Delta'$ as follows:

Given a $V$-maximal face $F \in \Delta$ and $u\in U$, let 
\begin{equation} \label{structure of G_i,j in (C, {C_u}_{u in U})^*}
\bigg\{ G_{u,1}^F, \ldots, G_{u, k_u^F}^F \bigg\} =\begin{cases}
\F(\Delta_{\C_u}), & \text{ if } u \notin F, \\\\
\F(\Delta_{\C_u^*}), & \text{ if } u \in F.
\end{cases}
\end{equation}
Then, we consider the block of sets associated to $F$ defined by
\[H_{j_U}^F := F \cup \bigcup_{u\in U}G_{u,j_u}^F,\]
where $j_U=(j_u\colon\ u\in U)$ and $j_u \in \{1,\ldots,k_u^F\}$, for all $u\in U$. We claim that every facet of $\Delta'$ is of the form $H_{j_U}^F$ for some $V$-maximal face $F\in\Delta$. First of all, note that for every $V$-maximal $F\in\Delta$ we have $F\cap V(\C_u^*)\in\Delta_{\C_u^*}$ if $u\in F$ and $F\cap V(\C_u)\in\Delta_{\C_u}$ if $u\notin F$ showing that the sets $H_{j_U}^F$ mentioned above are facets of $\Delta'$. In order to prove our claim, it is enough to show that every face $G \in \Delta'$ is contained in some $H_{j_U}^F$. 

Given $G \in \Delta'$, let us define
\begin{align*}
F_U &:= G \cap U, \\
F_V &:= G \cap V, \\
G_u &:= G \cap V(\C_u^*), \quad \text{if}\ u\in U\cap G,\\
G_u &:= G \cap V(\C_u), \quad \text{if}\ u\in U\setminus G.
\end{align*}
Since $G$ is an independent set in $\C'$, we conclude that $F:=F_U\cup F_V$ is an independent set in $\C$, and $G_u \in \Delta_{\C_u^*}$ if $u \in F$ and $G_u \in \Delta_{\C_u}$ if $u\notin F$. Hence our construction of $G_{u,1}^F,\ldots,G_{u,k_u^F}^F$ as in \eqref{structure of G_i,j in (C, {C_u}_{u in U})^*} guarantees the existence of $j_u \in \{1,\ldots,k_u^F\}$ such that $G_u \subseteq G_{u,j_u}^F$ for all $u\in U$. Thus
\[G= F \cup \bigcup_{u\in U}G_u\subseteq F \cup \bigcup_{u\in U}G_{u,j_u}^F \subseteq H_{j_U}^{F'},\]
where $F'\supseteq F$ is a $V$-maximal face of $\Delta$ satisfying $F'\cap U=F_U$.

(i) First observe that $\dim\Delta_{\C_u^*}\leq\dim\Delta_{\C_u} - 1$ for all $u\in U$. To see this, note that $\Delta_{\C_u^*}\subseteq\Delta_{\C_u}$. We show that $\F(\Delta_{\C_u^*})\cap\F(\Delta_{\C_u})=\varnothing$. Indeed, if $F$ is a common facet of $\Delta_{\C_u^*}$ and $\Delta_{\C_u}$, and $x\in V(\C_u)\setminus F$, then $F\cup\{x\}$ contains a circuit $e$ of $\C_u$ so that $F$ contains $e\setminus x$ and subsequently a circuit of $\C_u^*$, a contradiction. Thus $\dim\Delta_{\C_u^*}<\dim\Delta_{\C_u}$, as required. Now, we have
\begin{align*}
|H_{j_U}^F| &= |F\cap V| + \sum_{u\in U\cap F}(|F\cap V(\C_u^*)| + 1) + \sum_{u\in U\setminus F}|F\cap V(\C_u)|\\
&\leq (\dim\Delta_{\C[V]} + 1) + \sum_{u\in U\cap F}(\dim\Delta_{\C_u^*} + 2) + \sum_{u\in U\setminus F}(\dim\Delta_{\C_u} + 1)\\
&\leq |U| + \sum_{u\in U}\dim\Delta_{\C_u} + \dim\Delta_{\C[V]} + 1\\
&= |H_{j_U}^{F'}|,
\end{align*}
where $F'$ is a facet of maximum dimension in $\Delta_{\C[V]}$ and $G_{u,j_u}^{F'}$ is a facet of maximum dimension in $\Delta_{\C_u}$. It follows that 
\[\dim \Delta'  = |U| + \sum_{u\in U}\dim \Delta_{\C_u} + \dim \Delta_{\C[V]}.\]

(ii) It is straightforward.

(iii) Let $u\in U$. Then 
\begin{align*}
\Delta'\setminus u&=\gen{G_u\cup H_{j_{U\setminus\{u\}}}^F\colon\ u\in F\subseteq V(\C)\ \text{is $V$-maximal in}\ \Delta_\C,\ G_u\in\F(\Delta_{\C_u^*})\atop
G_u\cup H_{j_{U\setminus\{u\}}}^F\colon\ u\notin F\subseteq V(\C)\ \text{is $V$-maximal in}\ \Delta_\C,\ G_u\in\F(\Delta_{\C_u})}\\
&=\Delta_{\C_u}\star\Delta_{(\C\setminus u, \{\C_{u'}\}_{u'\in U\setminus\{u\}})^*},
\end{align*}
where the last equality follows form the fact that $\Delta_{\C_u^*}\subseteq\Delta_{\C_u}$, for all $u\in U$. This implies at once that every $u\in U$ is a shedding vertex of $\Delta'$ as every facet of $\Delta'\setminus u$ is also a facet of $\Delta'$.

On the other hand,
\begin{align*}
\link_{\Delta'}(u)&=\gen{H_{j_{U\setminus\{u\}}}^F\setminus\{u\}\colon\ u\in F\subseteq V(\C)\ \text{is $V$-maximal in}\ \Delta_\C}\\
&=\Delta_{\C_u^*}\star\Delta_{(\D_u,\{\C_{u'}\}_{u'\in U\setminus\{u\}})^*},
\end{align*}
where $\D_u=\min(\C / u)$. The rest of proof is entirely the same as in the proof of Theorem \ref{(C, (C_u)_(u in U))}. Notice that $\link_{\Delta'}(\cup_{u\in U}G_u)=\Delta_{\C[V]}$ when $G_u\in\F(\Delta_{\C_u})$, for all $u\in U$. To see this, assume that $G\in\Delta'$ contains $(\cup_{u'\in U}G_{u'})\cup\{u\}$ for some $u\in U$. If $x\in V(\C_u)\setminus G_u$, then $G_u\cup\{x\}$ contains a circuit $e$ of $\C_u$. By the definition, $e\setminus\{x\}$ contains a circuit $e'$ of $\C_u^*$, hence $G$ contains the circuit $\{u\}\cup e'$ of $\C'$, which is a contradiction.
\end{proof}

Notice that the independence complex of a complete graph is vertex-decomposable. Now, in view of Theorem \ref{(C, (C_u)_(u in U))^*}, we conclude that the graph $G'$ obtained form a graph $G$ by attaching a complete graph to each vertex of $G$ is vertex-decomposable. This covers the results of Hibi et. al. in Example \ref{examples to first construction}(i).
\section{Research problems}\label{Research problems}
In this section, we propose some research problems related to the contents of this paper. The first  problem concerns the proper independence property as we address below.

\subsection{PIP-triples}
The PIP-condition plays a central role in Theorem \ref{H'=(H,(U_i,D_i,H_i))} in order to describe the facets of the independence complex of hybrid hypergraphs. 
If $\S$ is a simplicial complex of dimension $d-1$, then the triple $(\alpha, \{d\}, \S)$ satisfies PIP for any $0<\alpha<d$ provided that $d=2$. Example \ref{Examples for PIP-triples} shows that the same statement is not true for $d=3$ in general. In this regard, we pose the following problem.
\begin{problem}
Let $d\geq3$ and $0<\alpha<d$ be positive integers. Find all simplicial complexes $\S$ of dimension $d-1$ such that the triple $(\alpha, \{d\}, \S)$ satisfies the proper independence property.
\end{problem}

Simple examples, as in Example \ref{Examples for PIP-triples}(i), show that a triple $(\alpha, D, \H)$ need not to satisfy the PIP-condition but some modifications of $\H$ may lead to a PIP-triple leaving some structural properties of $\H$ unchanged, for instance $\H^D$. Corollary \ref{Hybrid hypergraphs of designs} reveals the importance of this argument, hence suggests the following problem concerning the glued components of hybrid hypergraphs. 
\begin{problem}
Let $\H$ be a hypergraph, $D$ be a set of non-negative integers and $\alpha>0$. Under which conditions there exists a hypergraph $\H'$ satisfying $\H^D=\H'^D$ and $(\alpha, D, \H')$ is a PIP-triple.
\end{problem}

\subsection{Deletion-separations}
The numbers $\alpha(\U')+\wvd_f(\U')$ associated to $(W,f)$-deletions $\U'$ of a hypergraph $\U$ in Theorem \ref{H'=(H,(U_i,D_i,H_i))} can be described in a ``nice way" by a slight modification of weak vertex deletions. A strong vertex deletion $\U\setminus u$ is simply referred as deletion of $u$ from $\U$. While a weak vertex deletion $\U/u$ still removes $u$ from $\U$, adding $u$ back to $\U/u$ as an isolated vertex leads us to the notion of separation of $u$ from $\U$. Let $W$ be a set of vertices of $\U$ and $f\in\{0,1\}^W$ be a binary function. Then the deletion-separation $\DS_{(W,f)}(\U)$ of $\U$ is the hypergraph obtained from $\U$ by deleting every vertex $w\in W$ if $f(w)=0$ and separating every vertex $w\in W$ if $f(w)=1$ from $\U$. It follows that for a $(W,f)$-deletion $\U'$ of $\U$, the number $\alpha(\U')+\wvd_f(\U')$ is nothing but the independence number of $\DS_{(W,f)}(\U)$. Set
\[\DS_W(\U)=\set{\DS_{(W,f)}(\U)\colon\ f\in\{0,1\}^W}\]
and
\[\alpha(\DS_W(\U))=\set{\alpha(\U')\colon\ \U'\in \DS_W(\U)}\]
for every set $W$ of vertices of a hypergraph $\U$.

Simple examples (as Example \ref{DS_W1(U)=(2,...,n) and DS_W2(U)=(1,...n)}) illustrate that for a given hypergraph $\U$ one may find two strong dominating independent sets $W_1$ and $W_2$ such that $\DS_{W_1}(\U)\neq\DS_{W_2}(\U)$. This observation is very useful when we deal with part (iii) of Theorem \ref{H'=(H,(U_i,D_i,H_i))}. Moreover, this simple fact leads
to yet another interesting consequence that holds for all hypergraphs $\H$ admitting a PIP-triple $(\alpha, D, \H)$.
\begin{proposition}\label{H^(D-[0,x_1]) and H^(D-[0,x_2])}
Let $\H$ and $\U$ be hypergraphs and $D$ be a set of non-negative integers. Let $\mathcal{P}$ be the property of being (sequentially) Cohen-Macaulay, shellable, or vertex-decomposable. If $(\alpha(\U), D, \H)$ is a PIP-triple, and $W_1$ and $W_2$ are strong dominating independent sets in $\U$, then the following conditions are equivalent:
\begin{itemize}
\item[\rm (1)]$\H^{D-[0,\alpha(\U'_1)]}$ has the property $\mathcal{P}$ for all $\U'_1\in\DS_{W_1}(\U)$;
\item[\rm (2)]$\H^{D-[0,\alpha(\U'_2)]}$ has the property $\mathcal{P}$ for all $\U'_2\in\DS_{W_2}(\U)$.
\end{itemize}
\end{proposition}
\begin{example}\label{DS_W1(U)=(2,...,n) and DS_W2(U)=(1,...n)}
Let $\U$ be the hypergraph induced by edges $\{1,\ldots,n\}$ and $\{1,n+1\}$. Then $W_1:=\{1,\ldots,n-1\}$ and $W_2:=\{2,\ldots,n+1\}$ are strong dominating independent sets of $\U$ and that
\[\DS_{W_1}(\U)=\{2,\ldots,n\}\quad\text{and}\quad\DS_{W_2}(\U)=\{1,\ldots,n\}.\]
Let $\H$ be a hypergraph and $D$ be a set of integers such that $(n,D,\H)$ is a PIP-triple. It follows from Proposition \ref{H^(D-[0,x_1]) and H^(D-[0,x_2])} that $\H^{D-[0,1]}$ is sequentially Cohen-Macaulay/shellable/vertex-decomposable if $\H^{D-[0,i]}$ is so, for all $i=2,\ldots,n$.
\end{example}

Let $\alpha\geq1$ (possibly infinite) and $\Gamma_\alpha$ be the graph defined as follows: The vertices of $\Gamma_\alpha$ are finite subsets of $\NN$ and two vertices $X_1$ and $X_2$ are adjacent if there exists a hypergraph $\U$ with $\alpha(\U)\leq\alpha$ having two strong dominating independent sets $W_1$ and $W_2$ such that $X_1=\alpha(\DS_{W_1}(\U))$ and $X_2=\alpha(\DS_{W_2}(\U))$. Let $\mathcal{P}$ be the property of being sequentially Cohen-Macaulay, shellable, or vertex-decomposable. Proposition \ref{H^(D-[0,x_1]) and H^(D-[0,x_2])} states that if $X_1$ and $X_2$ are adjacent in $\Gamma_\alpha$, then $\H^{D-[0,x_1]}$ satisfies $\mathcal{P}$ for all $x_1\in X_1$ if and only if $\H^{D-[0,x_2]}$ satisfies $\mathcal{P}$ for all $x_2\in X_2$ provided that $(\max(X_1\cup X_2), D, \H)$ is a PIP-triple. It follows that if $(\alpha,D,\H)$ is a PIP-triple, then the following conditions are equivalent for any two vertices $X_1$ and $X_2$ in the same connected component of $\Gamma_\alpha$:
\begin{itemize}
\item[(1)]$\H^{D-[0,x_1]}$ satisfies $\mathcal{P}$ for all $x_1\in X_1$;
\item[(2)]$\H^{D-[0,x_2]}$ satisfies $\mathcal{P}$ for all $x_2\in X_2$.
\end{itemize}
Motivated by observations above we pose the following problem.
\begin{problem}
Determine the connected components of $\Gamma_\alpha$.
\end{problem}

In view of Corollary \ref{alpha(H/D)<=alpha(H)}, we know that $\alpha(\DS_{(W,f)}(\U))\leq\alpha(\U)$ for every deletion-separation $\DS_{(W,f)}(\U)$ of a hypergraph $\U$. The numbers $\alpha(\DS_{(W,f)}(\U))$ and their variety play a crucial role in part (iii) of Theorem \ref{H'=(H,(U_i,D_i,H_i))} in the sense that we have less requirements to check when $\DS_W(\U)$ is small. In particular, if $\alpha(\DS_{(W_i,f_i)}(\H[U_i]))=\alpha(\H[U_i])$ for all deletion-separations $\DS_{(W_i,f_i)}(\H[U_i])$ of $\H[U_i]$ and $1\leq i\leq m$, then in order to check the topological and combinatorial properties of the hybrid hypergraph $\H'=(\H,(U_i,D_i,\H_i)_{i=1}^m)$ in Theorem \ref{H'=(H,(U_i,D_i,H_i))}, we need only to check those properties for the hypergraphs $\H[V]$ and $\H_i^{D_i-[0,\alpha(\H[U_i])]}$, for $i=1,\ldots,m$. In this regard, the following problem turns out to be interesting.
\begin{problem}
Find all hypergraphs $\U$ having a strong dominating independent set $W$ for which $\alpha(\DS_{(W,f)}(\U))=\alpha(\U)$ for all $f\in\{0,1\}^W$.
\end{problem}

\subsection{Algebraic invariants}
In this paper, we have shown that for hypergraphs $\H'$ obtained by gluing a family of hypergraphs to a central one (in three different ways given in Theorems \ref{H'=(H,(U_i,D_i,H_i))}, \ref{(C, (C_u)_(u in U))}, \ref{(C, (C_u)_(u in U))^*}) the properties of being (sequentially) Cohen-Macaulay, shellable, and vertex-decomposable are transferred between the resulting hypergraphs $\H'$ and the glued hypergraphs. Indeed, the fact that facets of the hypergraphs $\H'$ are well known in all of our constructions in conjunction with \cite[Theorem 5.1.4]{wb-jh} enable us to compute the height $\mathrm{ht}(I(\H'))$ and the big height $\mathrm{bight}(I({\H'}))$ of $I(\H')$ by which we mean the minimum and maximum heights of associated primes of $I(\H')$. This shows that the depth of $\KK[\Delta_{\H'}]=S/I(\H')$ can also be computed when the glued hypergraphs are all sequentially Cohen-Macaulay. Actually, if $\Delta_{\H'}$ is sequentially Cohen-Macaulay, then \cite{rj-aayp} yields
\[\mathrm{depth}(\KK[\Delta_{\H'}])=\dim S - \mathrm{bight}(I(\H')).\]
\begin{problem}\ 
\begin{itemize}
\item[(1)]When is $\KK[\Delta_{\H'}]$ Gorenstein or Buchsbaum?
\item[(2)]What algebraic invariants of $\KK[\Delta_{\H'}]$ can be computed via those of the glued hypergraphs?
\end{itemize}
\end{problem}

\end{document}